\documentclass[11pt]{article}
\usepackage[totalwidth=13.0cm,totalheight=20.0cm]{geometry}
\usepackage{latexsym,amsthm,amsmath,amssymb,url}
\usepackage[ruled, linesnumbered]{algorithm2e}

\usepackage{tikz,authblk}
\usetikzlibrary{decorations.pathreplacing}

\newtheorem{lemma}{Lemma}

\newtheorem{theorem}{Theorem}

\DeclareMathOperator{\pfaffian}{pf}

\begin{document}
\title{Odd Properly Colored Cycles in Edge-Colored Graphs}

\date{}

\author[1]{Gregory Gutin}\author[1]{Bin Sheng}\author[1]{Magnus Wahlstr{\"o}m}

\affil[1]{Royal Holloway, University of London, TW20 0EX, Egham, Surrey, UK}
%\affil[2]{University of Haifa, Mount Carmel, Haifa, 3498838, Israel}

\maketitle

\pagestyle{plain}

\begin{abstract}

It is well-known that an undirected graph has no odd cycle if and only if it is bipartite. A less obvious, but similar result holds for directed graphs: a strongly connected digraph has no odd cycle if and only if it is bipartite. Can this result be further generalized to more general graphs such as edge-colored graphs?
In this paper, we study this problem and show how to decide if there exists an odd properly colored cycle in a given edge-colored graph. %Moreover, if the answer is positive, we will find one such cycle.
As a by-product, we show how to detect if there is a perfect matching in a graph with even (or odd) number of edges in a given edge set.

\end{abstract}

\section{Introduction}
A graph $G$ is {\em edge-colored} if each edge of $G$ is assigned a color. (Edge-colorings can be arbitrary, not necessarily proper.) A cycle $C$ in an edge-colored graph is {\em properly colored (PC)} if no pair of adjacent edges of $C$ have the same color. It is not hard to see that PC cycles in edge-colored graphs generalize directed cycles (dicycles) in digraphs: in a digraph $D$ replace every arc $uv$ by an undirected path $ux_{uv}v$, where $x_{uv}$ is a new vertex, and edges $ux_{uv}$ and  $x_{uv}v$ are of colors 1 and 2, respectively. Similarly, PC cycles generalize cycles in undirected graphs, e.g., by assigning every edge a distinct color. 

One of the central topics in graph theory is the existence of certain kinds of cycles in graphs. In digraphs, it is not hard to decide the existence of  any dicycle by simply checking whether a given digraph is acyclic~\cite{GG}. The problem of existence of PC cycles in edge-colored graphs is less trivial. To solve the problem, we may use Yeo's theorem~\cite{AY}: if an edge-colored graph $G$ has no PC cycle then $G$ contains a vertex $z$ such that no connected component of $G-z$ is joined to $z$ with edges of more than one color. Thus, we can recursively find such vertices $z$ and delete them from $G$; if we end up with a trivial graph (containing just one vertex) then $G$ has no PC cycles; otherwise $G$ has a PC cycle. Clearly, the recursive algorithm runs in polynomial time. 

One of the next natural questions is to decide whether a digraph has an odd (even, respectively) dicycle, i.e. a dicycle of odd (even) length, respectively. For odd dicycles we can employ the following well-known result (see, e.g.,~\cite{GG,HNC}): A strongly connected digraph is bipartite if and only if it has no odd dicycle. (Note that the result does not hold for non-strongly connected digraphs.) Thus, to decide whether a digraph $D$ has an odd cycle, we can find strongly connected components of $D$ and  check whether all components are bipartite. This leads to a simple polynomial-time algorithm. The question of whether we can decide in polynomial time whether a digraph has an even dicycle, is much harder and was an open problem for quite some time till it was solved, in affirmative, independently by McCuaig, and Robertson, Seymour and Thomas (see~\cite{RST}) who found highly non-trivial proofs. 

In this paper we consider the problem of deciding the existence of an odd PC cycle (and of finding one, if it exists) in an edge-colored graph in polynomial time. We show that while a natural extension of the odd dicycle solution does not work, an algebraic approach using Tutte matrices and the Schwartz-Zippel lemma allows us to prove that there is a randomized polynomial-time algorithm for solving the problem. The existence of a deterministic polynomial-time algorithm for the odd PC cycle problem remains an open question, as does the existence of a polynomial-time algorithm for the even PC cycle problem. 

In this paper, we allow our graphs to have multiple edges (but no loops) and call them, for clarity, {\em multigraphs}. In edge-colored multigraphs, we allow parallel edges of different colors (there is no need to consider parallel edges of the same color). For an edge $xy$ and a vertex $v$, we use $\chi(xy)$ and $\chi(v)$ to denote the color of $xy$ and the set of colors of edges incident to $v$, respectively.
For any other terminology and notation not provided here, we refer the readers to~\cite{GG}. There is an extensive literature on PC paths and cycles: for a detailed survey
of pre-2009 publications, see Chapter 16 of~\cite{GG}; more recent papers include~\cite{AD+,FM,Lo1,Lo2,Lo3}.

\section{Graph-Theoretical Approaches}

Recall that to solve the odd dicycle problem, in the previous section, we used the following result: A strongly connected digraph is bipartite if and only if it has no odd dicycle. It is not straightforward to generalize strong connectivity to edge-colored multigraphs. Indeed, color-connectivity\footnote{We will not define color-connectivity; an interested reader can find its definition in Sec. 16.6 of~\cite{GG}.}, introduced by Saad~\cite{Sa} under another name, does not appear to be useful to us as, in general, it does not partition vertices into components. Thus, we will use cyclic connectivity introduced by Bang-Jensen and Gutin~\cite{BJG} as follows. Let $P=\{H_1, \dots, H_p\}$ be a set of subgraphs of an edge colored multigraph $G$. The {\em intersection graph} $\Omega(P)$ of $P$ has the vertex set $P$ and the edge set $\{H_iH_j: V(H_i)\cap V(H_j)\neq \emptyset, 1\leq i < j\leq p\}$. A pair $x,y$ of vertices in an edge-colored multigraph $H$ is \textit{cyclic connected} if $H$ has a collection of PC cycles $P=\{C_1, \dots, C_p\}$ such that $x$ and $y$ belong to some cycles in $P$ and $\Omega(P)$ is a connected graph.  A maximum cyclic connected induced subgraph of $G$ is called a {\em cyclic connected component} of $G$. Note that cyclic connected components partition the vertices of $G$. Also note that cyclic connectivity for digraphs, where dicycles are considered instead of PC cycles, coincides with strong connectivity.  One could wonder whether every non-bipartite cyclic connected edge-colored graph has an odd PC cycle. Unfortunately, it is not true, see a graph $H$ in Fig. \ref{fig:counterExample}. It is not hard to check that $H$ is not bipartite and cyclic connected. It has even PC cycles, such as $v_1v_2v_5v_3v_1$, but no odd PC cycles.

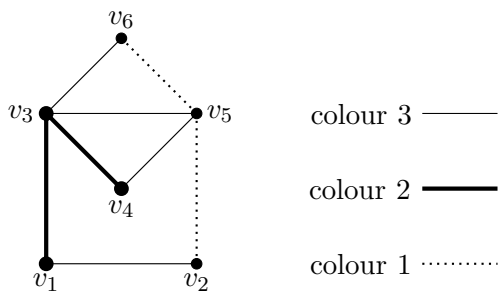
\begin{figure}\centering\label{figure1}
\begin{tikzpicture}
\draw [ultra thick](5,5)[fill]circle [radius=0.07]--(5,7)[fill]circle [radius=0.07]node [left]{$v_3$};
\draw (5,5) node [below]{$v_1$}--(7,5)[fill]circle [radius=0.07] node [below]{$v_2$};
\draw [ultra thick](5,7)--(6,6)[fill]circle [radius=0.07]node [below]{$v_4$};
\draw (6,6)--(7,7)[fill]circle [radius=0.07]node [right]{$v_5$};
\draw (5,7)--(6,8)[fill]circle [radius=0.07]node [above]{$v_6$};
\draw [dotted, thick](6,8)--(7,7);
\draw [dotted, thick](7,7)--(7,5);
\draw (5,7)--(7,7);

\draw [dotted, thick](10,5)node[left]{colour 1}  --(11,5) ;
\draw [ultra thick](10,6)node[left]{colour 2}  --(11,6) ;
\draw  (10,7) node[left]{colour 3}  --(11,7) ;
\end{tikzpicture}
\caption{Non-bipartite cyclic connected graph with no odd PC cycle}
  \label{fig:counterExample}
\end{figure}

%We do not know an explicit characterization of edge colored graphs that have no odd PC cycle.

Another natural idea is to find some odd PC closed walk first, and hope to find an odd PC cycle in it. 
Unfortunately,  we cannot generate all possible PC closed walks in polynomial time, and moreover a PC closed walk does not necessarily contain an odd PC cycle, see the graph in Figure \ref{fig:oddWalk}. It contains an odd PC walk, but not an odd PC cycle. 

\begin{figure}\centering\label{figure2}
\begin{tikzpicture}
\draw [ultra thick][fill] (0,0) circle [radius=0.07]node[left]{$v_1$}--(1,0)circle [radius=0.07]node[below]{$v_2$};
\draw (1,0)  --(2,1)[fill]circle [radius=0.07] node [below]{$v_5$};
\draw [ultra thick](2,1)  --(3,0) node [below]{$v_3$};
\draw (3,0)[fill]circle [radius=0.07]  --(4,0)[fill]circle [radius=0.07] node [below]{$v_4$};
\draw [dotted, thick](0,0)  --(0,1)[fill]circle [radius=0.07] node [left]{$v_0$};
\draw (0,1)  --(0,2)[fill]circle [radius=0.07] node [left]{$v_6$};
\draw [ultra thick](0,2)  --(1,2)[fill]circle [radius=0.07] node [above]{$v_7$};
\draw (1,2)  --(2,1) ;
\draw [ultra thick](2,1)  --(3,2)[fill]circle [radius=0.07] node [above]{$v_8$};
\draw (3,2)  --(4,2)[fill]circle [radius=0.07] node [right]{$v_9$};
\draw [ultra thick](4,0)  --(4,2) ;
\draw [dotted, thick](6,0)node[left]{colour 1}  --(7,0) ;
\draw [ultra thick](6,1)node[left]{colour 2}  --(7,1) ;
\draw  (6,2) node[left]{colour 3}  --(7,2) ;

\end{tikzpicture}

\caption{An odd PC closed walk}
  \label{fig:oddWalk}
\end{figure}
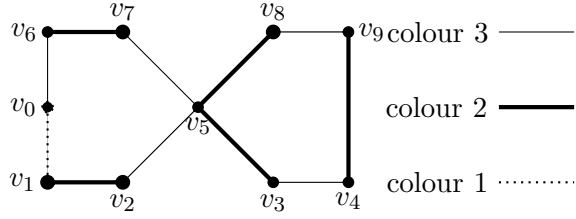

\section{Algebraic Approach}\label{sec3}

%It is not obvious how to solve the problem in polynomial time.  We now show how to detect if there is an odd PC cycle in an edge colored multi-graph randomly in polynomial time. 

In an edge-colored multigraph $G$, a vertex $v$ is {\em monochromatic} if $|\chi(v)|=1$. 
Let  $G'$ be the multigraph obtained from $G$ by recursively deleting  monochromatic vertices such that $G'$ has no monochromatic vertex. Following Szeider~\cite{Sz}, let $G_x$, $x\in V(G')$ denote a graph with vertex set $$\{x_i,x'_i:\ i\in \chi(x)\}\cup \{x''_a,x''_b\} \mbox{ and edge set }$$
$$\{x''_ax''_b,x'_ix''_a,x'_ix''_b:\ i\in \chi(x)\}\cup \{x_ix'_i:\ i\in \chi(x)\}.$$
Let $G^*$ denote a graph with vertex set $\bigcup_{x\in V(G')}V(G_x)$ and edge set $E_1\cup E_2$, where $E_1=\bigcup_{x\in V(G')}E(G_x)$ and $E_2=\{y_qz_q:\ yz\in E(G'), \chi(yz)=q\}.$ Let $c=\max\{\chi(x):\ x\in V(G)\}$. Note that \begin{equation}\label{eq0} |V(G^*)| = O(c|V(G)|) \end{equation}
 % Note that $|V(G^*)|<2cn$ and $|E(G^*)|= |E_1|+|E_2|< |E(G)|+3cn$. 
 
A subgraph of an edge-colored multigraph is called a \textit{PC cycle subgraph} if it consists of several vertex-disjoint PC cycles. We will use the following result of~\cite{GE}.

\begin{theorem}\label{thm2}
Let $G$ be a connected edge-colored multigraph such that $G'$
is non-empty and $G^*$ constructed as above. Then $G$ has a PC cycle subgraph with $r$ edges if and only if $G^{*}$ has a perfect matching with exactly $r$ edges in $E_2$. 
\end{theorem}

Using Theorem \ref{thm2}, the problem of deciding if there exists an odd PC cycle in $G$ reduces to that of deciding if there is a perfect matching with an odd number of edges from $E_2$ in the graph $G^{*}$
(indeed,~$G^*$ has an odd PC cycle subgraph if and only if it has an odd PC cycle).  
%Indeed, if there is a odd PC cycle, then there is a PC sub-graph with odd number of edges and therefore, by Theorem \ref{thm2}, there must be a perfect matching with odd number of edges from $E_2$ in graph %$G^{*}$.. On the other hand, if there is a perfect matching with odd number of edges in $E_2$, then there is a PC cycle sub-graph with odd number of edges, which implies there is an odd PC cycle in $G$.

We use the properties of Tutte matrices to solve the reduced problem. 
%\textbf{The Tutte matrix.} ~~~~~  
For a graph $G=(V,E)$ with $V=\{v_1,v_2,..., v_n\}$, the {\em Tutte matrix} $A_{G}$ is the $n \times n$ multivariate polynomial matrix with entries
\begin{equation}\label{eq1}
A_{G}(i,j) =
\left\{
\begin{array}{rl}
x_{ij} & \text{if } v_{i}v_{j}\in E ~~\text{and}~~i<j \\
-x_{ji} & \text{if } v_{i}v_{j}\in E ~~\text{and}~~i>j \\
0 & \text{otherwise,} 
\end{array}
\right.
\end{equation}
where $x_{ij}$ are distinct variables. Tutte~\cite{tt} proved that a graph $G$ has a perfect matching if and only if $\det A_{G}$ is not identically 0.

We say that a matrix $A$ is \textit{skew symmetric} if $A + A^{T}=0$. Note that the Tutte matrix is skew symmetric. In our argument, we will use the notion of Pfaffian of a skew symmetric matrix. 
Let $A = [a_{ij}]$ be a $2n \times 2n$ skew symmetric matrix. The {\em Pfaffian} of $A$ is defined as follows.
%by the following equation.
%$$
%\pfaffian  A =\frac{1}{2^n n!} \sum_{\sigma \in S_{2n}} sgn(\sigma) \prod_{i=1}^n a_{\sigma(2i-1),\sigma(2i)}
%$$
%where $S_{2n}$ is the symmetric group and $sgn(\sigma)$ is the signature of $\sigma$. 
%Alternatively, the Pfaffian can be computed more compactly as
\begin{equation} \label{eq:pf}
\pfaffian A = \sum_{\sigma} sgn(\sigma) \prod_{i=1}^n a_{\sigma(2i-1),\sigma(2i)},
\end{equation}
where $sgn(\sigma)$ is the signature of $\sigma$ and the summation is over all permutations~$\sigma$ such that~$\sigma(2i-1)<\sigma(2i)$ for each~$1 \leq i \leq n$,
and~$\sigma(2i)<\sigma(2i+2)$ for each~$1 \leq i < n$ (i.e., each partition of the set~$\{1,\ldots,2n\}$ into pairs is included in the sum exactly once). Note in particular that for a graph~$G$, this formula for~$\pfaffian A_G$ enumerates every perfect matching of~$G$ exactly once.

Observe that, if we regard $a_{ij}$ as indeterminate, $\pfaffian A$ is a multi-linear polynomial. For an odd skew symmetric matrix $A$, the Pfaffian is defined to be zero. We will use the following well-known relation between Pfaffian and determinants of skew symmetric matrices (see, e.g.,~\cite{LM}).

\begin{theorem}\label{thm3}
If $A$ is a skew symmetric matrix, then $\det A = (\pfaffian A)^2$.
\end{theorem}

Given a graph $G=(V,E)$ and a subset of edges $E_0\subseteq E(G)$, we now define another skew symmetric matrix $A_{G,E_0}$ whose entries are
\begin{equation}\label{eq2}
A_{G, E_0}(i,j) =
\left\{
\begin{array}{rl}
-A_{G}(i,j) & \text{if } v_{i}v_{j}\in  E_0 \\
A_{G}(i,j) & \text{if } v_{i}v_{j}\notin  E_0 \\ 
\end{array}
\right.
\end{equation}
It is easy to see that $A_{G, E_0}$ is also a skew symmetric matrix, thus by Theorem \ref{thm3}, $\det A_{G,E_0}= (\pfaffian A_{G,E_0})^2$. Note that $A_G$ and $A_{G,E_0}$ only differ at entries corresponding to edges in $E_0$. 
We call a perfect matching $M$ in a graph $G$ {\em $E_0$-odd} ({\em $E_0$-even}, respectively) if $|M\cap E_0|$ is odd (even, respectively). Here is our key result.

\begin{lemma}\label{lemma2}
Given a graph $G$ with even number of vertices and an edge subset $E_0\subseteq E(G)$, let $A_{G}$ and $A_{G,E_0}$ be defined as in (\ref{eq1}) and (\ref{eq2}). Then $\det A_{G,E_0}=\det A_{G}$ if and only if all the perfect matchings of $G$ are of same $E_0$-parity.
\end{lemma}

\begin{proof}
As both $A_{G}$ and $A_{G,E_0}$ are skew symmetric, by Theorem \ref{thm3}, 
$$\det A_G = (\pfaffian A_G)^2 \mbox{ and } \det A_{G,E_0}
=(\pfaffian A_{G,E_0})^2.$$ 
Thus $\det A_{G,E_0}=\det A_{G}$ if and only if $\pfaffian A_{G,E_0}=\pfaffian A_{G}$ or $\pfaffian A_{G,E_0}=-\pfaffian A_{G}$. 
By~(\ref{eq:pf}),~$\pfaffian A_G$ and~$\pfaffian A_{G,E_0}$ both enumerate perfect matchings of~$G$, 
and by the definitions of~$A_G$ and~$A_{G,E_0}$, for each such matching~$M$ its contributions to~$\pfaffian A_G$ and~$\pfaffian A_{G,E_0}$
differ (by a sign term) if and only if~$M$ is~$E_0$-odd. 
Hence 
%Thus by definitions of $A_{G}$ and $A_{G,E_0}$, we can get $\pfaffian A_{G,E_0}$ from $\pfaffian A_{G}$ by replacin each $x_{ij}$ with $-x_{ij}$ if and only if $v_{i}v_{j}\in E_0$. So 
$\pfaffian A_{G,E_0}=\pfaffian A_{G}$ ($\pfaffian A_{G,E_0}=-\pfaffian A_{G}$, respectively) if and only if each perfect matching in $G$ is $E_0$-even ($E_0$-odd, respectively). 
\end{proof}

For $G=G^*$ and $E_0=E_2\subseteq E(G^*)$, by Lemma \ref{lemma2}, if $\det {A^{*}}_{G,E_2}\neq \det A^{*}_{G}$, then $G^*$ has a $E_2$-odd perfect matching and a $E_2$-even perfect matching.
If $\det {A^{*}}_{G,E_2}=\det A^{*}_{G}$, then every perfect matching of $G^*$ is either $E_0$-even or $E_0$-odd. In such a case, we can 
%use Edmonds' algorithm~\cite{EJ} to 
find a perfect matching $M$ of the graph $G^*$ in polynomial time, and decide the parity of $M\cap E_2$. So we have an algorithm for deciding if $G^*$  has a perfect matching with even or odd number of edges in $E_2$. Unfortunately, we do not know whether this algorithm is polynomial or not as there is no polynomial algorithm to decide whether a multivariate polynomial is identically zero. Fortunately, we can have a polynomial randomized algorithm due to the following well-known lemma, called the Schwartz-Zippel lemma.\footnote{It was independently discovered by several authors: Schwartz~\cite{S}, Zippel~\cite{Z}, DeMillo and Lipton~\cite{DL}.}

\begin{lemma}\label{thm6}
Let $P(x_{1},x_{2},\dots,x_{n})$ be a multivariate polynomial of total degree at most $d$ over a field $F$, and assume that $P$ is not identically zero. Pick $r_{1},r_{2},\dots,r_{n}$ uniformly at random from a finite set $S$ of values where $S\subset F$. Then the probability
$\mathbb{P}(P(r_{1},r_{2},\dots,r_{n})=0)\leq \frac{d}{|S|}$.
\end{lemma}

%So we will use Theorem \ref{thm6} to decide if $\det{A^{*}}'_{G} =\det A^{*}_{G}$. If we claim that they are not equal, we know they are really different. If we claim they are equal, they may be equal, maybe not (with %small possibility). We still run the Edmonds' algorithm to find a perfect matching.
Now it is not hard to prove the following:

\begin{theorem}\label{thm7}
Let $G$ be an edge-colored multigraph with $n$ vertices and let  $c=\max\{\chi(v):\ v\in V(G)\}$. There is a randomized algorithm running in time $O((cn)^{\omega})$, where $\omega < 2.3729$, that decides if there is an odd PC cycle in $G$, with false negative probability less than 1/4.
\end{theorem}

\begin{proof}
As $f(\textbf{x})=\det A^{*}_{G,E_2} -\det A^{*}_{G}$ is a multivariate polynomial of degree at most $n$, we may choose a set $S$ of real values, such that $|S|> 4n$, and use Lemma \ref{thm6} to decide if $\det {A^{*}}_{G,E_2}\neq \det A^{*}_{G}$ in polynomial time, with false negative less than 1/4. 
%Pick $r_{1},r_{2},\dots,r_{n}$ uniformly at random from $S$. Then $\mathbb{P}(P(r_{1},r_{2},\dots,r_{n})=0)\leq 1/4$. False negative here means $f(\textbf{x})$ is not identically zero, but we answer it is. We may do the substitution enough times, such that the false probability is very small.
To see that the running time is $O((cn)^{\omega})$, recall (\ref{eq0}) and observe that computing the determinants of $A^{*}_{G,E_2}$ and $A^{*}_G$ will take time $O((cn)^\omega)$, where $\omega < 2.3729$, by the algorithm in~\cite{VVW}. 
Finally, for the case that $f(\textbf{x}) \equiv 0$, we can use the algorithm of Mucha and Sankowski~\cite{MS} to find a perfect matching in time~$O((cn)^\omega)$, and then decide its~$E_0$-parity.
 \end{proof}

\section{Conclusion and open problems}
After the publication of this paper \cite{GSW}, El Maalouly et al.~\cite{MSW} proved that Correct Parity Matching admits a deterministic
polynomial-time algorithm. Applying their result to the auxiliary graph $G^*$ constructed in Section~3, with $E_2$ as the distinguished edge set and odd parity prescribed, yields a deterministic polynomial-time algorithm for detecting and finding an odd properly colored cycle. Thus, the corresponding open problem in the original version of this paper has been resolved.

%We have proved that the problem can be solved in random polynomial time, then a following question is can this problem be solved in polynomial time? 
What about the problem of detecting even PC cycles? Note that there had been an extensive research on deciding if there is an even cycle in a digraph. It took a long time before people \cite{RST} finally gave a polynomial time algorithm to decide if there is an even cycle in a digraph. So presumably the problem of detecting even PC cycles should be hard. If we want to find an algorithm for even PC cycles, searching for a randomized polynomial algorithm would be a good start.

 \vspace{3mm}
 
 \noindent{\bf Acknowledgment.}  Research of BS was partially supported by China Scholarship Council.

\end{document}